\documentclass[11pt]{article}

\usepackage[utf8]{inputenc}
\usepackage[T1]{fontenc}
\usepackage{amsmath,amssymb,amsthm}
\usepackage{mathtools}
\usepackage{microtype}
\usepackage[margin=1in]{geometry}
\usepackage[colorlinks=true,linkcolor=blue,citecolor=blue,urlcolor=blue]{hyperref}
\hypersetup{
  pdftitle={Eventual Morse Minimality of Heat Flow on Lens Spaces},
  pdfauthor={Carlos A. Cadavid, Dairo J. Hern\'andez, and Juan Diego V\'elez Caicedo},
  pdfsubject={Heat flow, Morse functions, and lens spaces},
  pdfkeywords={heat equation, Morse functions, lens spaces, Laplace eigenfunctions, Morse--Bott theory, experimental mathematics, AI usage disclosure}
}

\newtheorem{theorem}{Theorem}
\newtheorem{lemma}{Lemma}
\newtheorem{proposition}{Proposition}
\newtheorem{definition}{Definition}
\newtheorem{remark}{Remark}

\title{Eventual Morse Minimality of Heat Flow on Lens Spaces}
\author{Carlos A. Cadavid \and Dairo J. Hern\'andez \and Juan Diego V\'elez Caicedo}

\date{\today}
\begin{document}

\maketitle
\begin{center}
{\small
Carlos A. Cadavid\\
Departamento de Ciencias Matem\'aticas, Universidad EAFIT, Medell\'in, Colombia\\
\texttt{ccadavid@eafit.edu.co}

\vspace{3mm}
Dairo J. Hern\'andez\\
Grupo GIFES, Facultad de Ingenier\'ia, Universidad de la Guajira, Riohacha, La Guajira, Colombia\\
\texttt{djhernandezp@uniguajira.edu.co}

\vspace{3mm}
Juan Diego V\'elez Caicedo\\
Departamento de Matem\'aticas, Universidad Nacional de Colombia, Medell\'in, Colombia\\
\texttt{jdvelez@unal.edu.co}
}
\end{center}

\noindent\textbf{Corresponding author.} Carlos A. Cadavid, \texttt{ccadavid@eafit.edu.co}.

\vspace{3mm}
\begin{abstract}
We prove that heat flow generically produces Morse-minimal functions on round
lens spaces. The result gives a rigorous explanation of a phenomenon first
suggested by computational experiments: after high-frequency modes have been
suppressed by diffusion, generic heat evolutions on lens spaces tend to settle
into Morse functions with the smallest possible number of critical points.
Precisely, for every lens space \(L(p,q)\) with \(p\geq2\),
\(1\leq q\leq p/2\), and \((p,q)=1\), there is an open dense set of initial data
\(f\in L^2(L(p,q),\mathbb R)\) such that the heat evolution \(e^{t\Delta}f\),
where \(\Delta\) is taken in the non-positive sign convention, is, for all
sufficiently large \(t\), a Morse function with exactly four critical points,
of indices \(0,1,2,3\). The proof analyzes the asymptotic spectral expansion of
the heat flow. In the most delicate case \(1<q<p/2\), the leading term is
Morse--Bott with two critical circles, and the higher heat modes break these
circles through a resonant Fourier mechanism. An arithmetic estimate shows that
the fundamental reduced frequency dominates all higher multiples, and generic
nonvanishing of the corresponding resonant coefficients gives the minimal
critical-point count.
\end{abstract}



\section{Introduction}

A smooth function on a closed manifold carries topological information through
its critical points. Morse theory makes this precise: the number and indices of
the critical points of a Morse function constrain, and are constrained by, the
topology of the underlying manifold. It is therefore natural to ask whether
there are canonical analytic procedures that produce topologically efficient
Morse functions.

In this article we study this question for heat flow on round lens spaces. Let
\[
L(p,q)=S^3/\Gamma,
\qquad
\gamma(\alpha,\beta)=(\zeta\alpha,\zeta^q\beta),
\qquad
\zeta=e^{2\pi i/p},
\]
where \(p\geq2\), \(1\leq q\leq p/2\), and \((p,q)=1\). We equip
\(L(p,q)\) with the metric induced from the round metric on \(S^3\). Given an
initial function \(f\in L^2(L(p,q),\mathbb R)\), we consider its heat evolution
\(u(t)=e^{t\Delta}f\).
As \(t\to\infty\), the heat flow suppresses high-frequency spectral components
faster than low-frequency ones. After subtracting the constant part and
rescaling, the long-time behavior is governed by the first nonzero spectral
component of \(f\), with the remaining components appearing as exponentially
small perturbations.

We use the non-positive Laplace--Beltrami convention throughout: if
\(\phi\in E_k\), then
\[
\Delta\phi=-\lambda_k\phi,
\qquad
\lambda_k=k(k+2).
\]
Thus the heat equation is \(\partial_tu=\Delta u\), and the heat semigroup is
\[
e^{t\Delta}\phi=e^{-\lambda_k t}\phi.
\]

This question belongs to a line of work in which heat flow is used as a
spectral selection mechanism for topologically efficient Morse functions. The
phenomenon was proved earlier for flat tori and round spheres
\cite{CadavidVelezHeatToriSpheres}, and later for real and complex projective
spaces \cite{MunozQuinteroProjective}. The present paper treats the first
nontrivial family of spherical space forms beyond projective space. On the
spectral side, our explicit use of lens-space harmonics is compatible with the
classical study of spectra of three-dimensional lens spaces
\cite{IkedaYamamotoLensSpectra} and with explicit eigenmode bases for lens and
prism spaces \cite{LehoucqUzanWeeks}. We use standard background on Morse and
Morse--Bott theory from \cite{MilnorMorse,Nicolaescu,BottMorse}, and standard
spectral facts for the round sphere and its quotients from \cite{Chavel,Besse}.

\subsection*{Experimental origin and guiding principle}

The theorem below was motivated by numerical experiments with heat flow and
spectral truncations on spherical space forms. Starting from random initial data
and evolving by diffusion, the observed critical-point counts on lens spaces
stabilized, after a transient regime, at the Morse-minimal value. These
experiments suggested that heat flow was not merely smoothing functions, but
selecting a topologically efficient Morse function from the spectral data of the
space.

The computations were especially useful in the cases in which the leading heat
term is Morse--Bott rather than Morse. They indicated that the critical circles
of the first nonconstant mode are broken by a specific Fourier frequency on the
circle, and that this frequency is the fundamental quotient frequency rather
than a higher multiple. This observation led to the resonant heat-weight analysis
used below: one first performs a Morse--Bott reduction, then identifies the
first possible reduced Fourier mode, and finally proves arithmetically that this
mode dominates all higher reduced frequencies. The proof in this paper is not
computer-assisted; the experiments serve as the discovery mechanism and as the
organizing guide for the formal argument.

The main result is that, for generic initial data, this analytic process
eventually produces a Morse-minimal function. More precisely, we prove the
following statement.

\begin{theorem}[Main theorem]\label{thm:main}
Let \(L(p,q)\) be a round lens space as above. Then there exists an open dense
subset \(\mathcal U_{p,q}\subset L^2(L(p,q),\mathbb R)\) such that, for every
\(f\in\mathcal U_{p,q}\), there is \(T_f>0\) with the following property: for
every \(t>T_f\), the heat flow \(u(t)=e^{t\Delta}f\) is Morse and has exactly
four critical points. Their Morse indices are \(0,1,2,3\). In particular,
\(u(t)\) is Morse-minimal for all sufficiently large \(t\).
\end{theorem}

The number four is the minimum possible for a Morse function on a nontrivial
lens space. We recall the standard argument, using basic homology and duality
for lens spaces; see, for example, \cite{Bredon,Rolfsen}. Choose a prime
\(\ell\mid p\). Then
\[
H_1(L(p,q);\mathbb F_\ell)\cong \mathbb F_\ell,
\qquad
H_2(L(p,q);\mathbb F_\ell)\cong \mathbb F_\ell,
\]
the second isomorphism following from Poincar\'e duality. Hence the Morse
inequalities with \(\mathbb F_\ell\)-coefficients give
\[
c_1\geq1,
\qquad
c_2\geq1.
\]
Since the manifold is closed and connected, every Morse function also has
\[
c_0\geq1,
\qquad
c_3\geq1.
\]
Thus every Morse function on \(L(p,q)\), \(p\ge2\), has at least four critical
points. Conversely, lens spaces admit genus-one Heegaard splittings, and hence
admit Morse functions with one critical point in each index. Thus four is the
Morse-minimal number.

The proof splits naturally into three cases.

For \(L(2,1)=\mathbb{RP}^3\), the first nonconstant eigenspace consists of
trace-free quadratic forms on
\(S^3\subset\mathbb R^4\). A generic such quadratic form has four critical
directions, giving four critical points on \(\mathbb{RP}^3\). Since the leading
term is already Morse, the conclusion follows by stability of Morse functions.

For \(L(p,1)\), \(p\geq3\), the first nonconstant eigenfunctions are pullbacks
of height functions on
\(S^2\) under the Hopf fibration. The leading term is therefore Morse--Bott,
with two critical Hopf circles. The first non-basic heat mode has vertical
frequency \(p\), and generically it breaks each critical circle into two
critical points on the quotient.

The most subtle case is \(p\geq3\), \(1<q<p/2\). Here the first nonconstant
eigenspace is one-dimensional, generated by
\[
\rho(\alpha,\beta)=|\alpha|^2-|\beta|^2.
\]
The function \(\rho\) is Morse--Bott, with two critical circles
\[
C_\alpha=\{\beta=0\},
\qquad
C_\beta=\{\alpha=0\}.
\]
The higher spectral components of the heat flow break these critical circles.
To understand the breaking, we perform Morse--Bott reduction near each circle,
obtaining a reduced one-dimensional function. The dominant reduced Fourier mode
is not obvious from the spectrum alone; it is determined by a resonance between
the spectral decay and the normal graph produced by the reduction. We introduce
heat weights
\[
\Omega_{\alpha,m},
\qquad
\Omega_{\beta,m},
\]
for the possible reduced frequencies \(mp\), and prove an arithmetic dominance
estimate:
\[
\Omega_{\alpha,1}<\Omega_{\alpha,m},
\qquad
\Omega_{\beta,1}<\Omega_{\beta,m},
\qquad
m\geq2.
\]
Thus the first reduced mode is the fundamental mode \(p\). Generic
nonvanishing of the associated resonant coefficients \(\mathfrak A_\alpha(f)\)
and \(\mathfrak A_\beta(f)\) then implies that each reduced function has the form
\[
c+2\operatorname{Re}(Ae^{ip\theta})+\text{lower-order terms},
\qquad
A\neq0.
\]
A Fourier-counting lemma gives \(2p\) critical points on each covering circle,
hence two critical points on each circle after quotienting by the order \(p\)
action. The two Morse--Bott circles therefore contribute four critical points in
total.

The next section gives a concise proof outline before the formal argument begins.

\section{Proof outline}

We recall the mechanism only once.  The heat expansion is
\[
u(t)=f_0+\sum_{k\ge1}e^{-\lambda_k t}f_k,
\qquad \lambda_k=k(k+2).
\]
After subtracting the constant term and rescaling by the first nonzero
eigenvalue, the heat flow converges in every \(C^r\)-norm to its leading
spectral component.  If this leading component is already Morse, stability of
Morse functions proves the result.  If the leading component is Morse--Bott,
we reduce the critical point equation near each critical circle to a
one-dimensional critical point problem on the circle.

From the experimental point of view, the key issue is not the spectral
convergence itself, but the geometry of the spectral component and perturbation
selected by that convergence.  In the lens spaces considered here, the leading
term is either Morse or Morse--Bott.  The Morse--Bott cases are where the
experimental evidence was most informative: it predicted the quotient frequency
that should appear after reduction to each critical circle, and the proof then
turns this prediction into the dominance estimates below.

There are three cases.  For \(L(2,1)=\mathbb{RP}^3\), the first nonconstant
component is a generic trace-free quadratic form on \(S^3\), hence descends to
a Morse function with four critical points.  For \(L(p,1)\), \(p\ge3\), the
low modes below degree \(p\) are Hopf-basic; the degree \(p\) component is the
first vertical oscillation and breaks each critical Hopf fiber into two points
on the quotient.  For \(1<q<p/2\), the leading component is
\[
\rho(\alpha,\beta)=|\alpha|^2-|\beta|^2,
\]
whose critical set is \(C_\alpha\cup C_\beta\).  The reduced circle
functions have only frequencies \(p,2p,3p,\ldots\).  The core arithmetic lemma
shows that the heat weight of the fundamental reduced frequency \(p\) is
strictly smaller than the heat weights of all higher multiples.  Generic
nonvanishing of the first resonant coefficient then reduces the problem to the
elementary fact that
\[
c+2\operatorname{Re}(Ae^{ip\theta})+o_{C^2}(|A|),
\qquad A\ne0,
\]
has exactly \(2p\) nondegenerate critical points on the covering circle, hence
two critical points after quotienting by the order-\(p\) action.

\section{Spectral setup and preliminary results}

Throughout, write
\[
L(p,q)=S^3/\Gamma,\qquad
\Gamma=\langle\gamma\rangle,\qquad
\gamma(\alpha,\beta)=(\zeta\alpha,\zeta^q\beta),\qquad
\zeta=e^{2\pi i/p},
\]
where
\[
p\ge 2,\qquad 1\le q\le p/2,\qquad (p,q)=1.
\]
Let
\[
S^3=\{(\alpha,\beta)\in\mathbb C^2:|\alpha|^2+|\beta|^2=1\}.
\]
For \(k\ge0\), let \(E_k\) denote the real eigenspace on \(L(p,q)\) coming from
\(\Gamma\)-invariant spherical harmonics of degree \(k\) on \(S^3\). Its positive spectral parameter is
\[
\lambda_k=k(k+2),
\]
meaning that
\[
\Delta\phi=-\lambda_k\phi
\qquad
\text{for }\phi\in E_k.
\]
In particular, the heat factor associated to \(E_k\) is \(e^{-\lambda_k t}\).

\begin{lemma}[Spectral tails and heat-weight estimates]
Let
\[
M=L(p,q)
\]
with the round quotient metric, and write the spectral expansion of
\[
f\in L^2(M,\mathbb R)
\]
as
\[
f=f_0+\sum_{k\geq1}f_k,
\qquad
f_k\in E_k.
\]
Let
\[
u(t)=e^{t\Delta}f
=
f_0+\sum_{k\geq1}e^{-\lambda_k t}f_k.
\]

Fix
\[
r\geq0.
\]
If
\[
K\geq1,
\]
then the spectral tail
\[
R_{\geq K}(t)=\sum_{k\geq K}e^{-\lambda_k t}f_k
\]
satisfies
\[
\|R_{\geq K}(t)\|_{C^r}
=
O(e^{-\lambda_K t})
\qquad
\text{as }t\to\infty.
\]
Equivalently, after normalizing by the degree \(2\) contribution, the tail
\[
\widetilde R_{\geq K}(t)
=
\sum_{k\geq K}e^{-(\lambda_k-\lambda_2)t}f_k
\]
satisfies
\[
\|\widetilde R_{\geq K}(t)\|_{C^r}
=
O(e^{-\delta_K t}),
\qquad
\delta_K=\lambda_K-\lambda_2.
\]

In particular, if
\[
\sigma<\delta_K,
\]
then
\[
\|\widetilde R_{\geq K}(t)\|_{C^r}
=
o(e^{-\sigma t}).
\]
More generally, whenever all spectral components occurring in a remainder have
heat weights at least
\[
\sigma+\eta
\]
for some
\[
\eta>0,
\]
the remainder is
\[
o(e^{-\sigma t})
\]
in every \(C^r\)-norm.
\end{lemma}

\begin{proof}
Choose
\[
s>r+\frac{\dim M}{2}.
\]
By Sobolev embedding,
\[
\|h\|_{C^r}\leq C\|h\|_{H^s}.
\]
Therefore
\[
\|R_{\geq K}(t)\|_{C^r}
\leq
C\|R_{\geq K}(t)\|_{H^s}.
\]
Using the spectral decomposition,
\[
\|R_{\geq K}(t)\|_{H^s}^2
\leq
C_s
\sum_{k\geq K}
(1+\lambda_k)^s e^{-2\lambda_k t}\|f_k\|_{L^2}^2.
\]

For \(t\) sufficiently large, the function
\[
\lambda\mapsto (1+\lambda)^s e^{-2\lambda t}
\]
is decreasing on the interval
\[
[\lambda_K,\infty).
\]
Indeed, its logarithmic derivative is
\[
\frac{s}{1+\lambda}-2t,
\]
which is negative for all
\[
\lambda\geq \lambda_K
\]
once \(t\) is large enough. Hence, for all \(k\geq K\),
\[
(1+\lambda_k)^s e^{-2\lambda_k t}
\leq
(1+\lambda_K)^s e^{-2\lambda_K t}.
\]
It follows that
\[
\|R_{\geq K}(t)\|_{H^s}^2
\leq
C_{s,K} e^{-2\lambda_K t}
\sum_{k\geq K}\|f_k\|_{L^2}^2
\leq
C_{s,K,f} e^{-2\lambda_K t}.
\]
Therefore
\[
\|R_{\geq K}(t)\|_{C^r}
=
O(e^{-\lambda_K t}).
\]

Multiplying by
\[
e^{\lambda_2t}
\]
gives the normalized estimate with
\[
\delta_K=\lambda_K-\lambda_2.
\]
The final assertion follows by applying the same estimate to the first spectral
degree whose heat weight is at least \(\sigma+\eta\).
\end{proof}

We shall use the following elementary convention. A phase weight \((m,n)\)
means that, under
\[
(\alpha,\beta)\mapsto(e^{i\phi_1}\alpha,e^{i\phi_2}\beta),
\]
the corresponding term transforms by the factor
\[
e^{i(m\phi_1+n\phi_2)}.
\]
Such a term is \(\Gamma\)-invariant precisely when
\[
m+qn\equiv0\pmod p.
\]

\subsection[Invariant spherical harmonics on S3]{Invariant spherical harmonics on \(S^3\)}

We shall use the following standard model for the eigenspaces.  Let
\(\mathcal P_k\) be the space of homogeneous polynomials of total degree
\(k\) in
\[
\alpha,
\overline\alpha,
\beta,
\overline\beta .
\]
The degree \(k\) spherical harmonics on \(S^3\) are the restrictions to
\(S^3\) of the harmonic elements of \(\mathcal P_k\).  Equivalently,
\(E_k\) is the \(\Gamma\)-invariant real part of this space.  If a monomial
has the form
\[
\alpha^{a_1}\overline\alpha^{a_2}
\beta^{b_1}\overline\beta^{b_2},
\]
then its phase weight is
\[
(a_1-a_2,\, b_1-b_2),
\]
and it is \(\Gamma\)-invariant precisely when
\[
(a_1-a_2)+q(b_1-b_2)\equiv0\pmod p.
\]
Moreover, any monomial in which each complex variable appears only
holomorphically or only anti-holomorphically is harmonic.  Indeed,
\[
\Delta_{\mathbb C^2}=4\left(
\frac{\partial^2}{\partial\alpha\partial\overline\alpha}
+
\frac{\partial^2}{\partial\beta\partial\overline\beta}
\right),
\]
and each mixed derivative above annihilates such a monomial.  This observation
will be used repeatedly to show that the minimal degrees appearing in the
weight computations are actually attained by genuine spherical harmonics, not
merely by formal monomials.

For later use, we record the basic Morse--Bott model that appears in the
non-Hopf case. Define
\[
\rho:S^3\to\mathbb R,
\qquad
\rho(\alpha,\beta)=|\alpha|^2-|\beta|^2.
\]
Its critical set is the union of the two Hopf circles
\[
C_\alpha=\{\beta=0\},
\qquad
C_\beta=\{\alpha=0\}.
\]
In coordinates near \(C_\alpha\),
\[
\alpha=\sqrt{1-|z|^2}e^{i\theta},
\qquad
\beta=z,
\]
one has
\[
\rho=1-2|z|^2.
\]
Thus \(C_\alpha\) has normal Morse index \(2\). In coordinates near
\(C_\beta\),
\[
\alpha=z,
\qquad
\beta=\sqrt{1-|z|^2}e^{i\theta},
\]
one has
\[
\rho=-1+2|z|^2,
\]
so \(C_\beta\) has normal Morse index \(0\).

\subsection{Morse--Bott reduction and Fourier counting}

\begin{lemma}[Normal reduction near the two critical circles]
Let \(F:S^3\to\mathbb R\) be smooth. Let \(U_\alpha,U_\beta\) be disjoint
tubular neighborhoods of \(C_\alpha,C_\beta\), with parametrizations

\[
\Phi_\alpha:S^1\times D^2\to U_\alpha,\qquad
\Phi_\beta:S^1\times D^2\to U_\beta,
\]
such that
\[
\Phi_\alpha(\theta,0)=(e^{i\theta},0),\qquad
\Phi_\beta(\theta,0)=(0,e^{i\theta}).
\]
There exist \(\delta>0\) and \(r>0\) such that if
\[
\|F-\rho\|_{C^2}<\delta,
\]
then there are unique smooth maps
\[
Y_\alpha,Y_\beta:S^1\to D^2
\]
with
\[
\|Y_\alpha\|_{C^1}<r,\qquad \|Y_\beta\|_{C^1}<r,
\]
such that the normal component of \(\nabla F\) vanishes along
\[
\theta\mapsto \Phi_\alpha(\theta,Y_\alpha(\theta)),
\qquad
\theta\mapsto \Phi_\beta(\theta,Y_\beta(\theta)).
\]
Define
\[
R_\alpha(\theta)=F(\Phi_\alpha(\theta,Y_\alpha(\theta))),
\qquad
R_\beta(\theta)=F(\Phi_\beta(\theta,Y_\beta(\theta))).
\]
Then the critical points of \(F\) in \(U_\alpha\) are exactly
\[
\Phi_\alpha(\theta,Y_\alpha(\theta))
\quad\text{with}\quad R_\alpha'(\theta)=0,
\]
and similarly near \(C_\beta\).
\end{lemma}

\begin{proof}
We prove the assertion near \(C_\alpha\). The proof near \(C_\beta\) is identical.
In the coordinates
\[
\alpha=\sqrt{1-|z|^2}e^{i\theta},\qquad \beta=z,
\]
one has
\[
\rho=1-2|z|^2.
\]
Hence the normal Hessian along \(C_\alpha\) is nondegenerate. For
\[
\widetilde F_\alpha(\theta,y)=F(\Phi_\alpha(\theta,y)),
\]
the normal equation is
\[
\partial_y\widetilde F_\alpha(\theta,y)=0.
\]
For \(F=\rho\), the solution is \(y=0\), and the derivative with respect to \(y\)
is uniformly invertible. The implicit function theorem with parameter \(\theta\)
gives the unique small solution \(y=Y_\alpha(\theta)\).

If \(x=\Phi_\alpha(\theta,y)\) is critical, then the normal equation holds, so
\(y=Y_\alpha(\theta)\), and differentiating the reduced function gives
\(R_\alpha'(\theta)=0\). Conversely, if \(R_\alpha'(\theta)=0\), the normal
component vanishes by construction and the tangential component vanishes by the
definition of \(R_\alpha\). Hence \(dF_x=0\).
\end{proof}

\begin{lemma}[Fourier counting lemma]
Let \(p\ge2\). There exists \(c_p>0\) such that if
\[
R(\theta)=c+2\operatorname{Re}(Ae^{ip\theta})+S(\theta),
\qquad A\ne0,
\]
and
\[
\|S\|_{C^2}<c_p|A|,
\]
then \(R\) has exactly \(2p\) critical points on \(S^1\), and all are
nondegenerate.
\end{lemma}

\begin{proof}
Write \(A=|A|e^{i\varphi}\). The unperturbed part
\[
R_0(\theta)=2|A|\cos(p\theta+\varphi)
\]
has critical points
\[
\theta_\ell=\frac{\ell\pi-\varphi}{p},\qquad \ell=0,\dots,2p-1,
\]
all nondegenerate. Choose \(0<\delta<\pi/4\) and intervals
\[
I_\ell=\left(\theta_\ell-\frac{\delta}{p},\theta_\ell+\frac{\delta}{p}\right).
\]
Outside their union,
\[
|R_0'|\ge 2p|A|\sin\delta.
\]
Inside each \(I_\ell\), the second derivative \(R_0''\) has constant sign and
\[
|R_0''|\ge 2p^2|A|\cos\delta.
\]
Choose
\[
c_p<\min\{p\sin\delta,p^2\cos\delta\}.
\]
Then \(R'=R_0'+S'\) has no zeros outside the \(I_\ell\)'s and exactly one zero
inside each \(I_\ell\), and at each such zero \(R''\ne0\).
\end{proof}

\begin{lemma}[Global localization of critical points]
Let \(U_\alpha,U_\beta\) be disjoint neighborhoods of \(C_\alpha,C_\beta\).
There exists \(\delta>0\) such that if
\[
\|F-\rho\|_{C^1}<\delta,
\]
then every critical point of \(F\) lies in
\[
U_\alpha\cup U_\beta.
\]
\end{lemma}

\begin{proof}
On the compact set \(K=S^3\setminus(U_\alpha\cup U_\beta)\), the gradient
\(\nabla\rho\) is nonzero. Hence \(|\nabla\rho|\ge m>0\) on \(K\). If
\(\|F-\rho\|_{C^1}<m/2\), then \(\nabla F\) cannot vanish on \(K\).
\end{proof}

\begin{remark}[Indices after Morse--Bott breaking]
A Morse--Bott critical circle of normal index \(\mu\) contributes indices
\[
\mu,\qquad \mu+1
\]
after reduction to a Morse function on the circle. Thus \(C_\beta\), of normal
index \(0\), contributes indices \(0,1\), while \(C_\alpha\), of normal index
\(2\), contributes indices \(2,3\). 
Thus, whenever the two reduced circle functions each contribute one minimum and
one maximum on the quotient, the resulting four critical points have indices
\[
0,\quad1,\quad2,\quad3.
\]
\end{remark}

\subsection[The case p >= 3, 1 < q < p/2]{The case \(p\ge3,\;1<q<p/2\)}

We now turn to the main case
\[
p\geq3,\qquad 1<q<p/2.
\]

In this range the first nonconstant eigenspace is one-dimensional:
\[
E_2=\mathbb R\rho,
\]
where \(\rho\) is the Morse--Bott function introduced above. Thus, after
subtracting the constant term and rescaling the heat flow, the leading term is
\(\rho\), whose critical set is
\[
C_\alpha\cup C_\beta.
\]
The purpose of this subsection is to understand how the higher heat modes break
these two critical circles.

The analysis has three steps. First, we perform Morse--Bott reduction near
\(C_\alpha\) and \(C_\beta\), reducing the critical point problem to a
one-dimensional Fourier problem on each circle. Second, we determine which
reduced Fourier mode is dominant as \(t\to\infty\). This is the arithmetic part
of the proof: although the reduced frequencies are \(p,2p,3p,\ldots\), the
fundamental frequency \(p\) has strictly smaller heat weight than all higher
multiples. Third, we show that the corresponding resonant coefficients are
generically nonzero. Once these three facts are established, the Fourier
counting lemma gives two critical points from each Morse--Bott circle on the
quotient lens space.

\begin{lemma}[Dominance of the fundamental reduced frequency]
Let
\[
p\geq 3,\qquad 1<a<p/2,\qquad (a,p)=1.
\]
For \(k\geq1\), write
\[
\delta_k=k(k+2)-8.
\]
For each integer \(m\geq1\), define
\[
\Omega_a(m)
=
\min_{d\in\mathbb Z}
\left[
\delta_{|mp-ad|+|d|}
+
|d|\delta_{a+1}
\right].
\]
Then
\[
\Omega_a(1)<\Omega_a(m)
\qquad
\text{for every }m\geq2.
\]
\end{lemma}

\begin{proof}
See Appendix~\ref{app:arithmetic-dominance}.
\end{proof}

This lemma will be applied twice: near \(C_\alpha\), with \(a=q\), and near
\(C_\beta\), with \(a=|\mu_\beta|\). Its role is to rule out the possibility
that a higher reduced frequency \(2p,3p,\ldots\) dominates the fundamental
frequency \(p\).

We now translate the preceding arithmetic estimate into the local
Morse--Bott reduction. The point is that a reduced Fourier mode is produced
after two operations: first one solves the normal critical point equation, and
then one substitutes the resulting normal graph into the original function.
Thus a reduced term carries both the heat weight of the harmonic monomial itself
and the heat weights of the normal graph factors that are substituted into it.

Before doing the arithmetic bookkeeping, we record the analytic justification
for retaining only finitely many weighted Taylor--Fourier contributions. In the
heat-flow applications below the spectral sum is infinite; for a prescribed
weight bound \(B\), we first keep all spectral components with relative heat
weight \(\le B\). The remaining heat tail is then \(o(e^{-Bt})\) in any fixed
\(C^r\)-norm by the smoothing estimates for the heat semigroup and the discrete
spectrum. The following finite-dimensional lemma is the local form of this
truncation.

\begin{lemma}[Weighted Morse--Bott expansion]\label{lem:weighted-morse-bott-expansion}
Let
\[
F_t=\rho+\sum_{j=1}^N e^{-\sigma_j t}\phi_j+R_t,
\qquad
0<\sigma_1<\cdots<\sigma_N,
\]
in a fixed tubular neighborhood of either \(C_\alpha\) or \(C_\beta\).  Fix
\(B>0\), and assume that \(R_t=o(e^{-Bt})\) in \(C^4\). Let
\(Y_t(\theta)\) be the small normal graph solving
\[
\partial_yF_t(\theta,Y_t(\theta))=0.
\]
Then, modulo \(o(e^{-Bt})\) in \(C^2\), uniformly along the circle, both
\(Y_t\) and the reduced function
\[
F_t^{\mathrm{red}}(\theta)=F_t(\theta,Y_t(\theta))
\]
are finite sums of weighted terms
\[
e^{-(n_1\sigma_1+\cdots+n_N\sigma_N)t}
\Psi_{n_1,\ldots,n_N}(\theta),
\qquad n_i\in\mathbb Z_{\ge0},
\]
with total weight
\[
n_1\sigma_1+\cdots+n_N\sigma_N\le B.
\]
The coefficients \(\Psi_{n_1,\ldots,n_N}\) are universal polynomial expressions
in finitely many jets of the \(\phi_j\)'s along the critical circle.
\end{lemma}

\begin{proof}
We give the proof near one critical circle; the two cases are identical. Write
local coordinates as \((\theta,y)\), where \(y\in\mathbb R^2\) is normal. Since
the circle is compact and the normal Hessian of \(\rho\) is nondegenerate along
it, after shrinking the tubular neighborhood we may write
\[
\partial_y\rho(\theta,y)=H(\theta)y+Q(\theta,y),
\qquad Q(\theta,y)=O(|y|^2),
\]
with \(H(\theta)\) uniformly invertible. Set
\(\varepsilon_j=e^{-\sigma_jt}\) and
\(\varepsilon=(\varepsilon_1,\ldots,\varepsilon_N)\), and first suppress the
remainder \(R_t\). The finite-parameter normal equation is
\[
\partial_y\left(\rho+\sum_{j=1}^N\varepsilon_j\phi_j\right)
(\theta,Y(\theta,\varepsilon))=0.
\]
The parameter-dependent implicit function theorem, applied uniformly in
\(\theta\), gives a unique small solution \(Y(\theta,\varepsilon)\), smooth in
\((\theta,\varepsilon)\), for \(|\varepsilon|\) small. Its Taylor expansion at
\(\varepsilon=0\) has the form
\[
Y(\theta,\varepsilon)
=
\sum_{|n|<M}\varepsilon^nY_n(\theta)
+O(|\varepsilon|^M)
\]
in \(C^3\), where
\(\varepsilon^n=\varepsilon_1^{n_1}\cdots\varepsilon_N^{n_N}\). Choose \(M\) so
large that \(M\sigma_1>B\). Then the Taylor remainder is \(o(e^{-Bt})\). Among
the finitely many monomials with \(|n|<M\), retain exactly those whose weighted
order
\[
|n|_\sigma=n_1\sigma_1+\cdots+n_N\sigma_N
\]
is at most \(B\). The remaining retained Taylor monomials have weights strictly
larger than \(B\); because there are only finitely many of them, their weights
are separated from \(B\) by a positive gap. Hence
\[
Y(\theta,e^{-\sigma_1t},\ldots,e^{-\sigma_Nt})
=
\sum_{|n|_\sigma\le B}e^{-|n|_\sigma t}Y_n(\theta)+o(e^{-Bt})
\]
in \(C^3\).

Now reintroduce \(R_t\). Let \(Y_t\) be the solution for the full function.
Set \(\varepsilon(t)=(e^{-\sigma_1t},\ldots,e^{-\sigma_Nt})\), and let
\(Y_t^0=Y(\theta,\varepsilon(t))\) be the solution with \(R_t=0\). The two
normal equations give
\[
A_t(\theta)(Y_t-Y_t^0)=-\partial_yR_t(\theta,Y_t),
\]
where
\[
A_t(\theta)=\int_0^1
\partial_y^2\left(\rho+\sum_{j=1}^N e^{-\sigma_jt}\phi_j\right)
\bigl(\theta,Y_t^0+s(Y_t-Y_t^0)\bigr)\,ds.
\]
The matrices \(A_t(\theta)\) converge to \(H(\theta)\) in \(C^1\), uniformly in
\(\theta\). Therefore \(A_t(\theta)\) is uniformly invertible for all large
\(t\). Since \(R_t=o(e^{-Bt})\) in \(C^4\), this gives
\[
Y_t-Y_t^0=o(e^{-Bt})
\]
in \(C^2\). Thus the normal graph has the asserted weighted expansion.

Finally substitute the expansion of \(Y_t\) into \(F_t(\theta,y)\). Taylor's
formula in the \(y\)-variables, together with the same finite weighted-order
truncation, gives a finite sum of terms
\[
e^{-|n|_\sigma t}\Psi_n(\theta),
\qquad |n|_\sigma\le B,
\]
modulo \(o(e^{-Bt})\) in \(C^2\). The coefficients are polynomial expressions in
the Taylor coefficients of the \(\phi_j\)'s and in the coefficients \(Y_n\). The
\(Y_n\)'s are obtained recursively by applying \(H(\theta)^{-1}\) to polynomial
expressions in jets of the \(\phi_j\)'s. Hence the coefficients are universal
polynomial expressions in finitely many jets of the \(\phi_j\)'s along the
critical circle. This proves the lemma.
\end{proof}

The following lemma records the resulting bookkeeping. For each reduced
frequency \(mp\), it gives the least possible heat weight near each of the two
Morse--Bott circles.

\begin{lemma}[Resonant heat weights]
Let
\[
p\geq3,\qquad 1<q<p/2,\qquad (p,q)=1.
\]
Write
\[
\delta_k=\lambda_k-\lambda_2=k(k+2)-8.
\]

Near \(C_\alpha\), use coordinates
\[
\alpha=\sqrt{1-|z|^2}e^{i\theta},
\qquad
\beta=z.
\]
The lowest normal-linear branch has the form
\[
e^{-iq\theta}z,
\]
appears first in degree \(q+1\), and has heat weight
\[
\delta_{q+1}.
\]
For \(m\geq1\) and \(d\in\mathbb Z\), define
\[
D_{\alpha,m}(d)=|mp-qd|+|d|
\]
and
\[
W_{\alpha,m}(d)
=
\delta_{D_{\alpha,m}(d)}
+
|d|\delta_{q+1}.
\]
Set
\[
\Omega_{\alpha,m}
=
\min_{d\in\mathbb Z}W_{\alpha,m}(d),
\]
and, in particular,
\[
\Omega_\alpha=\Omega_{\alpha,1},
\qquad
\mathcal D_\alpha
=
\{d\in\mathbb Z:W_{\alpha,1}(d)=\Omega_\alpha\}.
\]

Near \(C_\beta\), use coordinates
\[
\alpha=z,
\qquad
\beta=\sqrt{1-|z|^2}e^{i\theta}.
\]
Let \(\mu_\beta\) be the allowed normal-linear frequency of smallest absolute
value:
\[
q\mu_\beta+1\equiv0\pmod p,
\]
with
\[
|\mu_\beta|=\min\{|s|,|s-p|\},
\]
where \(s\) satisfies
\[
1<s<p-1,
\qquad
qs\equiv-1\pmod p.
\]
The lowest normal-linear branch near \(C_\beta\) has the form
\[
e^{i\mu_\beta\theta}z.
\]
Put
\[
\tau_\beta=-\mu_\beta.
\]
For \(m\geq1\) and \(d\in\mathbb Z\), define
\[
D_{\beta,m}(d)=|mp-\tau_\beta d|+|d|
\]
and
\[
W_{\beta,m}(d)
=
\delta_{D_{\beta,m}(d)}
+
|d|\delta_{|\mu_\beta|+1}.
\]
Set
\[
\Omega_{\beta,m}
=
\min_{d\in\mathbb Z}W_{\beta,m}(d),
\]
and, in particular,
\[
\Omega_\beta=\Omega_{\beta,1},
\qquad
\mathcal D_\beta
=
\{d\in\mathbb Z:W_{\beta,1}(d)=\Omega_\beta\}.
\]

Then
\[
\Omega_{\alpha,1}<\Omega_{\alpha,m},
\qquad
\Omega_{\beta,1}<\Omega_{\beta,m}
\]
for every
\[
m\geq2.
\]
\end{lemma}

\begin{proof}
We begin near \(C_\alpha\).

A local Taylor--Fourier monomial near \(C_\alpha\) has the form
\[
e^{i\ell\theta}z^a\overline z^{\,b},
\qquad
a,b\geq0.
\]
The generator acts by
\[
(\theta,z)\mapsto
\left(\theta+\frac{2\pi}{p},\zeta^qz\right).
\]
Thus \(\Gamma\)-invariance forces
\[
\ell+q(a-b)\equiv0\pmod p.
\]
Let
\[
d=a-b.
\]
The leading normal graph associated to the lowest normal-linear branch has phase
\[
z\sim e^{iq\theta},
\qquad
\overline z\sim e^{-iq\theta}.
\]
After substitution, the monomial contributes phase
\[
e^{i(\ell+qd)\theta}.
\]

To produce reduced frequency \(mp\), we must have
\[
\ell+qd=mp.
\]
Hence
\[
\ell=mp-qd.
\]
For fixed \(d\), the smallest possible normal order \(a+b\) with \(a-b=d\) is
\[
|d|.
\]
Therefore the smallest possible spherical degree of a term with resulting
frequency \(mp\) and normal difference \(d\) is
\[
|\ell|+|d|
=
|mp-qd|+|d|
=
D_{\alpha,m}(d).
\]
Such a term is realized by the harmonic monomial
\[
\alpha^{\ell_+}\overline\alpha^{\,\ell_-}
\beta^{d_+}\overline\beta^{\,d_-},
\]
where
\[
\ell_+=\max\{\ell,0\},\quad
\ell_-=\max\{-\ell,0\},
\]
and
\[
d_+=\max\{d,0\},\quad d_-=\max\{-d,0\}.
\]
Its total degree is \(D_{\alpha,m}(d)\), and it is \(\Gamma\)-invariant because
\[
\ell+qd=mp\equiv0\pmod p.
\]

The coefficient of such a monomial has heat weight
\[
\delta_{D_{\alpha,m}(d)}.
\]
Substituting the leading normal graph contributes \(|d|\) factors, each with
weight
\[
\delta_{q+1}.
\]
Therefore the least possible heat weight for frequency \(mp\) and normal
difference \(d\) is
\[
W_{\alpha,m}(d)
=
\delta_{D_{\alpha,m}(d)}
+
|d|\delta_{q+1}.
\]
Minimizing over \(d\in\mathbb Z\) gives
\[
\Omega_{\alpha,m}.
\]

We now justify that no other contribution produced by the normal reduction has
smaller heat weight.

The normal graph is obtained by solving the normal critical point equation by
the implicit function theorem. Consequently it has an expansion in the spectral
coefficients of the perturbation. Its lowest nonzero term is produced by the
lowest normal-linear branch
\[
e^{-iq\theta}z,
\]
and has heat weight
\[
\delta_{q+1}.
\]
Every other term in the normal graph has strictly larger heat weight. Indeed,
such a term either comes from a higher normal-linear branch, hence from an
eigenspace of degree strictly larger than \(q+1\), or from a nonlinear expression
in previously determined terms of the normal graph. In the latter case its heat
weight is a sum of at least two positive heat weights, and is therefore strictly
larger than the lowest weight \(\delta_{q+1}\).

Thus, if one substitutes into a Taylor--Fourier monomial and uses any branch of
the normal graph other than the leading branch, the resulting heat weight is
strictly larger than the weight obtained by using only the leading branch.

It remains to check that adding extra normal factors cannot lower the weight.
For fixed
\[
d=a-b,
\]
the smallest possible normal order is
\[
a+b=|d|.
\]
Any other monomial with the same value of \(d\) has normal order
\[
|d|+2j,
\qquad
j\geq1.
\]
Such an extra factor \(z\overline z\) does not change the resulting Fourier
phase, but it increases the spherical degree from
\[
|mp-qd|+|d|
\]
to
\[
|mp-qd|+|d|+2j,
\]
and it also requires two additional factors of the normal graph. Since
\[
\delta_k=k(k+2)-8
\]
is strictly increasing for \(k\geq2\), both effects strictly increase the heat
weight.

Therefore, for a fixed resulting reduced frequency \(mp\), the least possible
heat weight is exactly
\[
W_{\alpha,m}(d)
=
\delta_{D_{\alpha,m}(d)}
+
|d|\delta_{q+1},
\]
minimized over \(d\in\mathbb Z\). Hence
\[
\Omega_{\alpha,m}
\]
is indeed the first possible heat weight for the reduced frequency \(mp\) near
\(C_\alpha\).

Now apply the dominance of the fundamental reduced frequency lemma with
\[
a=q.
\]
Since
\[
1<q<p/2,\qquad (q,p)=1,
\]
that lemma gives
\[
\Omega_{\alpha,1}<\Omega_{\alpha,m}
\qquad
\text{for every }m\geq2.
\]
Thus the first possible nonconstant reduced frequency near \(C_\alpha\) is
\(p\), and its first heat weight is
\[
\Omega_\alpha=\Omega_{\alpha,1}.
\]

The proof near \(C_\beta\) is the same, with \(q\) replaced by
\[
\tau_\beta=-\mu_\beta.
\]

Indeed, a local Taylor--Fourier monomial near \(C_\beta\) has the form
\[
e^{i\ell\theta}z^a\overline z^{\,b},
\qquad
a,b\geq0.
\]
Let
\[
d=a-b.
\]
The leading normal graph near \(C_\beta\) has phase
\[
z\sim e^{i\tau_\beta\theta},
\qquad
\overline z\sim e^{-i\tau_\beta\theta},
\]
where
\[
\tau_\beta=-\mu_\beta.
\]
Therefore, after substituting the leading normal graph, the monomial contributes
phase
\[
e^{i(\ell+\tau_\beta d)\theta}.
\]
To produce reduced frequency \(mp\), we must have
\[
\ell+\tau_\beta d=mp.
\]
Thus
\[
\ell=mp-\tau_\beta d.
\]

For fixed \(d\), the smallest possible normal order is \(|d|\), and therefore
the smallest possible spherical degree is
\[
|\ell|+|d|
=
|mp-\tau_\beta d|+|d|
=
D_{\beta,m}(d).
\]
The corresponding coefficient has heat weight
\[
\delta_{D_{\beta,m}(d)}.
\]
Substituting the leading normal graph contributes \(|d|\) factors of the lowest
normal graph, each of heat weight
\[
\delta_{|\mu_\beta|+1}.
\]
Thus the least possible heat weight associated to \(d\) is
\[
W_{\beta,m}(d)
=
\delta_{D_{\beta,m}(d)}
+
|d|\delta_{|\mu_\beta|+1}.
\]

As near \(C_\alpha\), using a higher branch of the normal graph strictly
increases the heat weight. Adding extra factors \(z\overline z\) also strictly
increases the heat weight without changing the resulting Fourier phase. Hence,
after minimizing over \(d\in\mathbb Z\),
\[
\Omega_{\beta,m}
=
\min_{d\in\mathbb Z}W_{\beta,m}(d)
\]
is indeed the first possible heat weight for the reduced frequency \(mp\) near
\(C_\beta\).

It remains only to relate this expression to the preceding arithmetic lemma.
Since
\[
\tau_\beta=-\mu_\beta,
\]
we have
\[
|\tau_\beta|=|\mu_\beta|.
\]
Moreover, because the minimum is taken over all
\[
d\in\mathbb Z,
\]
the sign of \(\tau_\beta\) is irrelevant: replacing \(d\) by \(-d\), if
necessary, gives
\[
\min_{d\in\mathbb Z}
\left[
\delta_{|mp-\tau_\beta d|+|d|}
+
|d|\delta_{|\mu_\beta|+1}
\right]
=
\min_{d\in\mathbb Z}
\left[
\delta_{|mp-|\mu_\beta|d|+|d|}
+
|d|\delta_{|\mu_\beta|+1}
\right].
\]
Thus the arithmetic lemma applies with
\[
a=|\mu_\beta|.
\]

We now check its hypotheses. Since
\[
q\mu_\beta+1\equiv0\pmod p,
\]
the integer \(\mu_\beta\) is invertible modulo \(p\). Therefore
\[
(|\mu_\beta|,p)=1.
\]
Moreover, \(|\mu_\beta|\neq1\). If \(\mu_\beta=1\), then
\[
q+1\equiv0\pmod p,
\]
which is impossible under \(1<q<p/2\). If \(\mu_\beta=-1\), then
\[
-q+1\equiv0\pmod p,
\]
which would force \(q=1\), again impossible. Finally, by the definition of
\(\mu_\beta\) as the representative of smallest absolute value,
\[
|\mu_\beta|<p/2.
\]
Thus
\[
1<|\mu_\beta|<p/2,
\qquad
(|\mu_\beta|,p)=1.
\]
The dominance of the fundamental reduced frequency lemma therefore gives
\[
\Omega_{\beta,1}<\Omega_{\beta,m}
\qquad
\text{for every }m\geq2.
\]
This proves the lemma.

\end{proof}

Thus, after Morse--Bott reduction, the first possible nonconstant reduced
frequency is the fundamental frequency \(p\), both near \(C_\alpha\) and near
\(C_\beta\). The next task is to identify the corresponding coefficients and
show that they are generically nonzero.

We now isolate the coefficient of the dominant reduced mode. Since there may be
several resonant Taylor--Fourier terms with the same minimal heat weight, the
coefficient is defined as a total coefficient: one sums all contributions of
heat weight \(\Omega_\alpha\) near \(C_\alpha\), and all contributions of heat
weight \(\Omega_\beta\) near \(C_\beta\).

\begin{definition}[First resonant coefficients]
With the notation of the resonant heat weights lemma, work after the
normalization
\[
U(t)=a^{-1}e^{\lambda_2t}(u(t)-f_0)=\rho+\text{positive heat-weight terms},
\]
so that every nonconstant spectral component has heat weight
\[
\delta_k=\lambda_k-\lambda_2.
\]
For a fixed bound \(B\), the weighted Morse--Bott expansion lemma expresses the
normal graph and the reduced function, modulo \(o(e^{-Bt})\), as a finite
weighted Taylor--Fourier expansion. We call the individual contributions in this
finite expansion \emph{weighted Taylor--Fourier trees}: this is only a
bookkeeping term for the finitely many expressions produced by iterating the
implicit-function expansion, not an additional combinatorial structure. Each
such contribution has a resulting Fourier frequency on the reduced circle, a
total heat weight, and a coefficient which is a universal polynomial in finitely
many normalized spectral coefficients of \(f\).

Near \(C_\alpha\), define
\[
\mathfrak A_\alpha(f)
\]
to be the sum of the coefficients of all weighted Taylor--Fourier trees whose
resulting reduced frequency is \(+p\) and whose total heat weight is exactly
\[
\Omega_\alpha=\Omega_{\alpha,1}.
\]
Equivalently, \(\mathfrak A_\alpha(f)\) is the total coefficient of
\(e^{ip\theta}\) in the part of the reduced function near \(C_\alpha\) having
heat weight exactly \(\Omega_\alpha\). This is the same as summing all minimal
contributions indexed by
\[
d\in\mathcal D_\alpha
=\{d\in\mathbb Z: W_{\alpha,1}(d)=\Omega_\alpha\}.
\]

Similarly, define
\[
\mathfrak A_\beta(f)
\]
to be the sum of the coefficients of all weighted Taylor--Fourier trees whose
resulting reduced frequency is \(+p\) and whose total heat weight is exactly
\[
\Omega_\beta=\Omega_{\beta,1}.
\]
Equivalently, \(\mathfrak A_\beta(f)\) is the total coefficient of
\(e^{ip\theta}\) in the part of the reduced function near \(C_\beta\) having
heat weight exactly \(\Omega_\beta\), or the sum of the minimal contributions
indexed by
\[
d\in\mathcal D_\beta
=\{d\in\mathbb Z: W_{\beta,1}(d)=\Omega_\beta\}.
\]

Thus
\[
\mathfrak A_\alpha,
\qquad
\mathfrak A_\beta
\]
are finite-dimensional polynomial functions of the relevant normalized spectral
coefficients of \(f\). More precisely, \(\mathfrak A_\alpha\) depends only on the
lowest normal-linear coefficient near \(C_\alpha\) and on the projections of
\(f\) to the eigenspaces
\[
E_{D_{\alpha,1}(d)},
\qquad
d\in\mathcal D_\alpha,
\]
together with the finitely many lower-weight spectral components that enter the
normal graph at total weight at most \(\Omega_\alpha\). Likewise,
\(\mathfrak A_\beta\) depends only on the lowest normal-linear coefficient near
\(C_\beta\), on the projections of \(f\) to
\[
E_{D_{\beta,1}(d)},
\qquad
d\in\mathcal D_\beta,
\]
and on the finitely many lower-weight spectral components entering the normal
graph at total weight at most \(\Omega_\beta\). In particular, after the leading
coefficient \(a\) of \(f_2=a\rho\) is fixed nonzero, the two quantities above
are honest polynomial functions of finitely many normalized spectral
coordinates. As functions of the original spectral coordinates, they become
polynomial after multiplication by a suitable power of \(a\).
\end{definition}

We now prove that the two resonant coefficients introduced above are not
artifacts of notation. Each of them is a genuine, nonzero polynomial function of
the relevant spectral data. This is the algebraic input that later makes the
nonvanishing conditions generic.

\begin{lemma}[Nontriviality of the first resonant coefficients]
Let
\[
p\geq3,\qquad 1<q<p/2,\qquad (p,q)=1.
\]
The first resonant coefficients are not identically zero:
\[
\mathfrak A_\alpha\not\equiv0,
\qquad
\mathfrak A_\beta\not\equiv0.
\]
\end{lemma}

\begin{proof}
We prove the statement for \(\mathfrak A_\alpha\). The proof for
\(\mathfrak A_\beta\) is analogous and is given afterward.

Near
\[
C_\alpha=\{\beta=0\},
\]
write
\[
\alpha=\sqrt{1-|z|^2}e^{i\theta},
\qquad
\beta=z.
\]
The lowest normal-linear branch is
\[
e^{-iq\theta}z.
\]
It is realized by the harmonic monomial
\[
\overline{\alpha}^{\,q}\beta.
\]
Indeed, this monomial is harmonic because it is holomorphic in \(\beta\) and
anti-holomorphic in \(\alpha\), and it is \(\Gamma\)-invariant because its
weight is
\[
-q+q=0.
\]
Thus the real-valued function
\[
L_\alpha
=
2\operatorname{Re}\left(A\,\overline{\alpha}^{\,q}\beta\right),
\qquad A\in\mathbb C,\ A\neq0,
\]
has local leading term
\[
2\operatorname{Re}\left(Ae^{-iq\theta}z\right).
\]

Choose
\[
d\in\mathcal D_\alpha,
\]
so that
\[
W_{\alpha,1}(d)=\Omega_\alpha.
\]
Set
\[
\ell=p-qd.
\]
Then
\[
\ell+qd=p\equiv0\pmod p.
\]
Write
\[
\ell_+=\max\{\ell,0\},
\qquad
\ell_-=\max\{-\ell,0\},
\]
and
\[
d_+=\max\{d,0\},
\qquad
d_-=\max\{-d,0\}.
\]
Consider the monomial
\[
M_{\alpha,d}
=
\alpha^{\ell_+}
\overline{\alpha}^{\,\ell_-}
\beta^{d_+}
\overline{\beta}^{\,d_-}.
\]
This monomial is harmonic on \(\mathbb C^2\), because in each complex variable
it involves either the variable or its conjugate, but not both. Its total degree
is
\[
|\ell|+|d|
=
|p-qd|+|d|
=
D_{\alpha,1}(d).
\]
It is \(\Gamma\)-invariant because its weight is
\[
\ell+qd=p\equiv0\pmod p.
\]
Near \(C_\alpha\), it has leading term
\[
M_{\alpha,d}
=
e^{i\ell\theta}z^{d_+}\overline z^{\,d_-}
+
O(|z|^{|d|+2}).
\]

Now consider the two-parameter family
\[
F_{\varepsilon,\eta}
=
\rho
+
\varepsilon L_\alpha
+
2\operatorname{Re}\left(\eta M_{\alpha,d}\right),
\]
where
\[
\varepsilon>0
\]
is small and
\[
\eta\in\mathbb C
\]
is small. Let
\[
R_{\varepsilon,\eta}(\theta)
\]
be the reduced function obtained by Morse--Bott reduction near \(C_\alpha\).

First put \(\eta=0\). The normal equation for
\[
\rho+\varepsilon L_\alpha
\]
has leading solution
\[
z_\varepsilon(\theta)
=
\frac{\varepsilon}{2}\overline A\,e^{iq\theta}
+
O(\varepsilon^2).
\]
Indeed,
\[
\rho=1-2|z|^2+O(|z|^4),
\]
and the normal-linear term is
\[
2\operatorname{Re}\left(Ae^{-iq\theta}z\right).
\]

We now vary the reduced function in a fixed complex direction
\[
\dot\eta\in\mathbb C.
\]
Since
\[
z_\varepsilon(\theta)
\]
solves the normal equation for
\[
\rho+\varepsilon L_\alpha,
\]
the envelope identity gives the first variation
\[
\left.
\frac{d}{ds}
R_{\varepsilon,s\dot\eta}(\theta)
\right|_{s=0}
=
2\operatorname{Re}
\left(
\dot\eta\,
M_{\alpha,d}(\theta,z_\varepsilon(\theta))
\right).
\]
Equivalently, the complex coefficient of \(e^{ip\theta}\) in this first
variation is the complex coefficient of \(e^{ip\theta}\) in
\[
\dot\eta\,
M_{\alpha,d}(\theta,z_\varepsilon(\theta)).
\]

Substituting the leading normal graph, we obtain
\[
M_{\alpha,d}(\theta,z_\varepsilon(\theta))
=
C_d\,\varepsilon^{|d|}
e^{i(\ell+qd)\theta}
+
O(\varepsilon^{|d|+2}),
\]
where
\[
C_d\neq0.
\]
Indeed, if \(d\geq0\), then
\[
z_\varepsilon^d
=
\left(\frac{\varepsilon}{2}\overline A\right)^d
e^{iqd\theta}
+
O(\varepsilon^{d+2}),
\]
whereas if \(d<0\), then
\[
\overline z_\varepsilon^{-d}
=
\left(\frac{\varepsilon}{2}A\right)^{-d}
e^{iqd\theta}
+
O(\varepsilon^{-d+2}).
\]
In both cases the phase is
\[
e^{i(\ell+qd)\theta}
=
e^{ip\theta}.
\]
Choosing
\[
\dot\eta
\]
so that
\[
\dot\eta C_d\neq0,
\]
we obtain a nonzero \(e^{ip\theta}\)-coefficient. Therefore the coefficient of
\(e^{ip\theta}\) in the part of heat weight
\[
\Omega_\alpha
\]
is not identically zero as a polynomial in the relevant spectral coefficients.

Therefore
\[
\mathfrak A_\alpha\not\equiv0.
\]

We now prove the corresponding statement for \(\mathfrak A_\beta\). Near
\[
C_\beta=\{\alpha=0\},
\]
write
\[
\alpha=z,
\qquad
\beta=\sqrt{1-|z|^2}e^{i\theta}.
\]
Let \(\mu_\beta\) be the lowest normal-linear frequency, so that
\[
q\mu_\beta+1\equiv0\pmod p.
\]
The lowest normal-linear branch has the form
\[
e^{i\mu_\beta\theta}z.
\]

This branch is realized by the harmonic monomial
\[
N_\beta=
\begin{cases}
\beta^{\mu_\beta}\alpha,
&
\mu_\beta>0,
\\[0.4em]
\overline\beta^{-\mu_\beta}\alpha,
&
\mu_\beta<0.
\end{cases}
\]
Indeed, near \(C_\beta\) this monomial has leading term
\[
e^{i\mu_\beta\theta}z.
\]
It is harmonic because, in each complex variable, it involves either the
variable or its conjugate, but not both. It is \(\Gamma\)-invariant because its
weight is
\[
q\mu_\beta+1\equiv0\pmod p.
\]

Let
\[
L_\beta=2\operatorname{Re}(B N_\beta),
\qquad B\in\mathbb C,\ B\neq0.
\]
The leading normal graph determined by
\[
\rho+\varepsilon L_\beta
\]
has phase
\[
z\sim e^{-i\mu_\beta\theta}.
\]
Put
\[
\tau_\beta=-\mu_\beta.
\]

Choose
\[
d\in\mathcal D_\beta
\]
and set
\[
\ell=p-\tau_\beta d.
\]
Then
\[
\ell+\tau_\beta d=p.
\]
Define
\[
\ell_+=\max\{\ell,0\},
\qquad
\ell_-=\max\{-\ell,0\},
\]
and
\[
d_+=\max\{d,0\},
\qquad
d_-=\max\{-d,0\}.
\]
Consider
\[
M_{\beta,d}
=
\beta^{\ell_+}
\overline{\beta}^{\,\ell_-}
\alpha^{d_+}
\overline{\alpha}^{\,d_-}.
\]
This monomial is harmonic, again because each complex variable appears either
holomorphically or anti-holomorphically, but not both.

It is \(\Gamma\)-invariant because its weight is
\[
q\ell+d.
\]
Using
\[
\ell=p-\tau_\beta d
\]
and
\[
q\tau_\beta\equiv1\pmod p,
\]
we get
\[
q\ell+d
\equiv
-q\tau_\beta d+d
\equiv
-d+d
\equiv0
\pmod p.
\]

Near \(C_\beta\), the monomial has leading term
\[
M_{\beta,d}
=
e^{i\ell\theta}z^{d_+}\overline z^{\,d_-}
+
O(|z|^{|d|+2}).
\]
Substituting the leading normal graph
\[
z\sim e^{i\tau_\beta\theta}
\]
gives phase
\[
e^{i(\ell+\tau_\beta d)\theta}
=
e^{ip\theta},
\]
with nonzero coefficient.

Therefore the same envelope-variation argument used near \(C_\alpha\) shows
that the \(e^{ip\theta}\)-coefficient in heat weight
\[
\Omega_\beta
\]
is not identically zero. Hence
\[
\mathfrak A_\beta\not\equiv0.
\]

The proof of the lemma is complete.
\end{proof}

The preceding weighted expansion lemma is the analytic device behind the
following uses of ``first heat weight''. Once a threshold \(\sigma\) is fixed,
all contributions of total heat weight strictly larger than \(\sigma\) are
\(o(e^{-\sigma t})\) in \(C^2\), and only finitely many spectral components and
Taylor terms can contribute at weights \(\le\sigma\). Thus the coefficient of a
given first heat weight is a finite algebraic expression, not a formal infinite
sum. This is the only role played below by the terminology of
weighted Taylor--Fourier trees.

The resonant heat-weight computation identifies the first possible reduced
frequency. We now use it to obtain the actual asymptotic form of the reduced
functions. The point is that the coefficient of the fundamental mode \(p\) is
asymptotic to the first resonant coefficient, while all remaining Fourier modes
are smaller in \(C^2\). This is precisely the form needed for the Fourier
counting lemma.

\begin{lemma}[Dominance of the first resonant harmonic]
Let \(p\ge3\), \(1<q<p/2\), and \((p,q)=1\). Let
\(f=f_0+\sum_{k\ge1}f_k\) be the spectral decomposition of \(f\). Assume
\(f_2=a\rho\), with \(a\ne0\),
that the lowest normal-linear coefficients near \(C_\alpha\) and \(C_\beta\)
are nonzero, and that \(\mathfrak A_\alpha(f)\ne0\) and
\(\mathfrak A_\beta(f)\ne0\). Let \(u(t)=e^{t\Delta}f\) and
\(U(t)=a^{-1}e^{\lambda_2t}(u(t)-f_0)\). Let \(R_{\alpha,t}\) and
\(R_{\beta,t}\) be the reduced functions obtained by Morse--Bott reduction near
\(C_\alpha\) and \(C_\beta\), respectively. Then
\[
R_{\alpha,t}(\theta)
=
c_{\alpha,t}
+
2\operatorname{Re}\left(A_{\alpha,t}e^{ip\theta}\right)
+
S_{\alpha,t}(\theta),
\]
and
\[
R_{\beta,t}(\theta)
=
c_{\beta,t}
+
2\operatorname{Re}\left(A_{\beta,t}e^{ip\theta}\right)
+
S_{\beta,t}(\theta),
\]
where \(A_{\alpha,t}\) and \(A_{\beta,t}\) are the total complex coefficients of
the Fourier mode \(e^{ip\theta}\) in the corresponding reduced functions, and
\[
A_{\alpha,t}\sim \mathfrak A_\alpha(f)e^{-\Omega_\alpha t},
\qquad
A_{\beta,t}\sim \mathfrak A_\beta(f)e^{-\Omega_\beta t}.
\]
Moreover,
\[
\frac{\|S_{\alpha,t}\|_{C^2}}{|A_{\alpha,t}|}\to0,
\qquad
\frac{\|S_{\beta,t}\|_{C^2}}{|A_{\beta,t}|}\to0.
\]
Consequently, for all sufficiently large \(t\), each reduced function has
exactly \(2p\) nondegenerate critical points.
\end{lemma}

\begin{proof}
We prove the statement near \(C_\alpha\). The proof near \(C_\beta\) is the
same, using the corresponding notation.

After normalization,
\[
U(t)=a^{-1}e^{\lambda_2t}(u(t)-f_0),
\]
we have
\[
U(t)=\rho+r(t),
\qquad
\|r(t)\|_{C^2}\to0.
\]
Moreover, each spectral component of degree \(k\) appears with heat weight
\[
\delta_k=\lambda_k-\lambda_2.
\]

The normal reduction near \(C_\alpha\) gives a reduced function
\[
R_{\alpha,t}:S^1\to\mathbb R.
\]
Since \(U(t)\) is \(\Gamma\)-invariant, the reduced function is invariant under
the induced rotation
\[
\theta\mapsto \theta+\frac{2\pi}{p}.
\]
Therefore its nonconstant Fourier frequencies are
\[
\pm p,\ \pm2p,\ \pm3p,\ldots.
\]

Let
\[
A_{\alpha,t}
\]
denote the total complex coefficient of \(e^{ip\theta}\) in
\(R_{\alpha,t}\). Thus \(A_{\alpha,t}\) includes not only the terms of minimal
heat weight \(\Omega_\alpha\), but also all higher-weight contributions to the
same Fourier frequency \(p\).

By definition of the first resonant coefficient, the part of \(A_{\alpha,t}\)
having heat weight exactly \(\Omega_\alpha\) is
\[
\mathfrak A_\alpha(f)e^{-\Omega_\alpha t}.
\]
Every other contribution to the same Fourier coefficient \(e^{ip\theta}\) has
strictly larger heat weight. Hence
\[
A_{\alpha,t}
=
\mathfrak A_\alpha(f)e^{-\Omega_\alpha t}
+
o(e^{-\Omega_\alpha t}).
\]
Since
\[
\mathfrak A_\alpha(f)\ne0,
\]
we get
\[
A_{\alpha,t}\ne0
\]
for all sufficiently large \(t\), and
\[
|A_{\alpha,t}|
\sim
|\mathfrak A_\alpha(f)|e^{-\Omega_\alpha t}.
\]

Now define
\[
S_{\alpha,t}
\]
to be the sum of all Fourier modes of \(R_{\alpha,t}\) except the constant mode
and the modes
\[
\pm p.
\]
Thus \(S_{\alpha,t}\) contains only frequencies
\[
\pm mp,
\qquad
m\ge2.
\]

It is enough to consider the positive frequencies \(mp\), since the reduced
functions are real-valued and the coefficients of the modes \(-mp\) are the
complex conjugates of those of \(mp\).

By the minimum resonant heat weight lemma, the first possible heat weight for
the reduced frequency \(mp\) is
\[
\Omega_{\alpha,m}.
\]
By the dominance of the fundamental reduced frequency,
\[
\Omega_{\alpha,m}>\Omega_{\alpha,1}=\Omega_\alpha
\qquad
\text{for every }m\ge2.
\]
Therefore every Fourier mode appearing in \(S_{\alpha,t}\) has heat weight
strictly larger than
\[
\Omega_\alpha.
\]
By the finite-weight observation above, the contribution
of all terms of heat weight strictly larger than \(\Omega_\alpha\) is
\[
o(e^{-\Omega_\alpha t})
\]
in \(C^2\). Hence
\[
\|S_{\alpha,t}\|_{C^2}
=
o(e^{-\Omega_\alpha t}).
\]
Consequently,
\[
\frac{\|S_{\alpha,t}\|_{C^2}}
{|A_{\alpha,t}|}
\longrightarrow0.
\]

This proves the claimed asymptotic form near \(C_\alpha\):
\[
R_{\alpha,t}(\theta)
=
c_{\alpha,t}
+
2\operatorname{Re}
\left(
A_{\alpha,t}e^{ip\theta}
\right)
+
S_{\alpha,t}(\theta),
\]
with
\[
A_{\alpha,t}
\sim
\mathfrak A_\alpha(f)e^{-\Omega_\alpha t},
\qquad
\frac{\|S_{\alpha,t}\|_{C^2}}
{|A_{\alpha,t}|}
\to0.
\]

The argument near \(C_\beta\) is identical. Let
\[
A_{\beta,t}
\]
be the total complex coefficient of \(e^{ip\theta}\) in the reduced function
\(R_{\beta,t}\). Its minimal heat-weight contribution is
\[
\mathfrak A_\beta(f)e^{-\Omega_\beta t}.
\]
All other contributions to the same frequency \(p\) have strictly larger heat
weight. Hence
\[
A_{\beta,t}
=
\mathfrak A_\beta(f)e^{-\Omega_\beta t}
+
o(e^{-\Omega_\beta t}).
\]
Since
\[
\mathfrak A_\beta(f)\ne0,
\]
we have
\[
|A_{\beta,t}|
\sim
|\mathfrak A_\beta(f)|e^{-\Omega_\beta t}.
\]

Let
\[
S_{\beta,t}
\]
denote the sum of all Fourier modes of \(R_{\beta,t}\) except the constant mode
and the modes
\[
\pm p.
\]
Then \(S_{\beta,t}\) contains only frequencies
\[
\pm mp,
\qquad
m\ge2.
\]

It is enough to consider the positive frequencies \(mp\), since the reduced
functions are real-valued and the coefficients of the modes \(-mp\) are the
complex conjugates of those of \(mp\).

By the minimum resonant heat weight lemma and the dominance of the fundamental
reduced frequency,
\[
\Omega_{\beta,m}>\Omega_{\beta,1}=\Omega_\beta
\qquad
\text{for every }m\ge2.
\]
Therefore every Fourier mode appearing in \(S_{\beta,t}\) has heat weight
strictly larger than
\[
\Omega_\beta.
\]

By the finite-weight observation above, the contribution of all terms of heat
weight strictly larger than \(\Omega_\beta\) is
\[
o(e^{-\Omega_\beta t})
\]
in \(C^2\).

Hence
\[
\|S_{\beta,t}\|_{C^2}
=
o(e^{-\Omega_\beta t}).
\]
Consequently,
\[
\frac{\|S_{\beta,t}\|_{C^2}}
{|A_{\beta,t}|}
\to0.
\]

For all sufficiently large \(t\), both reduced functions therefore satisfy the
hypotheses of the Fourier counting lemma. Hence each reduced function has
exactly
\[
2p
\]
nondegenerate critical points.
\end{proof}

We can now finish the analysis of the case \(1<q<p/2\). The normalized heat
flow converges to the Morse--Bott function \(\rho\), whose critical set is the
union of the two circles \(C_\alpha\) and \(C_\beta\). Global localization
confines all critical points to small tubular neighborhoods of these circles,
and the preceding resonance analysis shows that each reduced one-dimensional
function has exactly \(2p\) nondegenerate critical points. Passing to the
quotient by \(\Gamma\) gives two critical points from each circle.

\begin{proposition}[Heat flow for \(1<q<p/2\)]
Let
\[
p\geq3,\qquad 1<q<p/2,\qquad (p,q)=1,
\]
and let
\[
L(p,q)=S^3/\Gamma,
\qquad
\gamma(\alpha,\beta)=(\zeta\alpha,\zeta^q\beta).
\]
Let
\[
u(t)=e^{t\Delta}f
\]
be the heat flow, and write
\[
f=f_0+\sum_{k\geq1}f_k
\]
for the spectral decomposition of \(f\).

Assume
\[
f_2=a\rho,\qquad a\neq0,
\]
where
\[
\rho(\alpha,\beta)=|\alpha|^2-|\beta|^2.
\]
Assume also that the lowest normal-linear coefficients near \(C_\alpha\) and
\(C_\beta\) are nonzero, and that the first resonant coefficients satisfy
\[
\mathfrak A_\alpha(f)\neq0,
\qquad
\mathfrak A_\beta(f)\neq0.
\]

Then there exists
\[
T_f>0
\]
such that, for every
\[
t>T_f,
\]
the function
\[
u(t):L(p,q)\to\mathbb R
\]
is Morse and has exactly four critical points. Their Morse indices are
\[
0,\quad1,\quad2,\quad3.
\]
\end{proposition}

\begin{proof}
We lift the heat flow to \(S^3\). Constants do not affect critical points, so
we subtract \(f_0\), and then normalize by setting
\[
U(t)=a^{-1}e^{\lambda_2t}(u(t)-f_0).
\]
Then
\[
U(t)=\rho+r(t),
\qquad
\|r(t)\|_{C^2}\to0.
\]

The critical set of \(\rho\) is
\[
C_\alpha\cup C_\beta,
\]
where
\[
C_\alpha=\{\beta=0\},
\qquad
C_\beta=\{\alpha=0\}.
\]
By the global localization lemma, for all sufficiently large \(t\), every
critical point of \(U(t)\) lies in a small tubular neighborhood of
\[
C_\alpha\cup C_\beta.
\]

Near each critical circle we apply the Morse--Bott normal reduction. By the
dominance of the first resonant harmonic, the reduced functions near
\(C_\alpha\) and \(C_\beta\) satisfy the hypotheses of the Fourier counting
lemma for all sufficiently large \(t\). Hence each reduced function has exactly
\[
2p
\]
nondegenerate critical points.

The normal Hessian of \(\rho\) is nondegenerate transverse to both critical
circles. Therefore the critical points obtained from the reduced functions lift
to nondegenerate critical points of \(U(t)\). Thus \(U(t)\) has exactly
\[
2p+2p=4p
\]
critical points on \(S^3\), all nondegenerate.

Since \(u(t)\) is \(\Gamma\)-invariant, the set of critical points of its lift is
\(\Gamma\)-invariant. The action of \(\Gamma\) on \(S^3\) is free, so these
\(4p\) critical points split into orbits of cardinality \(p\). Hence the
descended function on \(L(p,q)\) has exactly
\[
4
\]
critical points.

It remains to record the indices. For the normalized function \(U(t)\), the
circle \(C_\beta\) has normal index \(0\), and its reduced critical points have
circle indices \(0\) and \(1\). Thus it contributes indices
\[
0,\quad1.
\]
The circle \(C_\alpha\) has normal index \(2\), and therefore contributes
indices
\[
2,\quad3.
\]
Thus the indices for \(U(t)\) are
\[
0,\quad1,\quad2,\quad3.
\]

Finally,
\[
u(t)-f_0
=
a e^{-\lambda_2t}U(t).
\]
If
\[
a>0,
\]
the Morse indices of \(u(t)\) agree with those of \(U(t)\). If
\[
a<0,
\]
they are replaced by
\[
i\longmapsto 3-i.
\]
In either case the four indices are exactly
\[
0,\quad1,\quad2,\quad3.
\]
This proves the proposition.
\end{proof}

\subsection[The case q = 1, p >= 3]{The case \(q=1,\;p\ge3\)}

We next consider the lens spaces
\[
L(p,1),\qquad p\geq3.
\]
This case is simpler than the general case \(1<q<p/2\), because the cyclic
action is along the Hopf fibers. As a result, all invariant eigenfunctions of
degree \(k<p\) are Hopf-basic. The leading nonconstant heat term is therefore
the pullback of a height function on \(S^2\). Its critical set on \(S^3\)
consists of two Hopf circles.

The first non-basic oscillation occurs at degree \(p\). Restricted to either
critical Hopf circle, this component has vertical Fourier frequency \(p\). For
generic data the corresponding coefficient is nonzero on both critical fibers.
Thus each Morse--Bott circle is broken into \(2p\) critical points on \(S^3\),
which descend to two critical points on the quotient \(L(p,1)\).

\begin{proposition}[Heat flow for \(L(p,1)\)]
Let
\[
p\geq3,
\qquad
L(p,1)=S^3/\Gamma,
\qquad
\gamma(\alpha,\beta)=(\zeta\alpha,\zeta\beta).
\]
Let
\[
u(t)=e^{t\Delta}f
\]
be the heat flow, and write
\[
f=f_0+\sum_{k\geq1}f_k
\]
for the spectral decomposition of \(f\).

Assume
\[
f_2\neq0.
\]
Then \(f_2\) is the pullback, by the Hopf map
\[
\pi:S^3\to S^2,
\]
of a nonzero linear function on \(S^2\). Let
\[
C_-,
\qquad
C_+
\]
be the two Hopf fibers over its critical points. Assume that the restriction of
\(f_p\) to each of the two fibers \(C_-\) and \(C_+\) has nonzero vertical
frequency \(p\).

Then there exists
\[
T_f>0
\]
such that, for every
\[
t>T_f,
\]
the function
\[
u(t):L(p,1)\to\mathbb R
\]
is Morse and has exactly four critical points. Their Morse indices are
\[
0,\quad1,\quad2,\quad3.
\]
\end{proposition}

\begin{proof}
We lift the heat flow to \(S^3\). Constants do not affect critical points, so
we subtract \(f_0\) and normalize by setting
\[
U(t)=e^{\lambda_2t}(u(t)-f_0).
\]
Then
\[
U(t)=f_2+r(t),
\qquad
\|r(t)\|_{C^2}\to0.
\]

Since \(q=1\), the action of \(\Gamma\) is along the Hopf fibers. For every
degree
\[
k<p,
\]
a \(\Gamma\)-invariant harmonic of degree \(k\) has Hopf frequency
\[
0.
\]
Thus all invariant eigenspaces of degree \(k<p\) are Hopf-basic.

The leading term \(f_2\) is therefore the pullback of a nonzero linear function
on \(S^2\). Its critical set on \(S^3\) is the union of two Hopf fibers,
\[
C_-\cup C_+.
\]
This is a Morse--Bott critical set. The normal index is \(0\) along the fiber
over the minimum of the base height function, and \(2\) along the fiber over the
maximum.

By Morse--Bott localization, for all sufficiently large \(t\), every critical
point of \(U(t)\) lies in a small tubular neighborhood of
\[
C_-\cup C_+.
\]
Near each critical Hopf fiber, we apply Morse--Bott normal reduction.

A small point must be kept explicit.  The Hopf-basic components of degrees
\(3,\ldots,p-1\) do not create vertical oscillation, but they do move the two
critical points of the base function on \(S^2\).  Let
\[
h_t=f_2+\sum_{3\le k<p}e^{-(\lambda_k-\lambda_2)t}f_k
\]
be the Hopf-basic part below degree \(p\), viewed as a function on \(S^2\).
Since \(f_2\) is a Morse height function, its two critical points persist as
smooth points
\[
x_-(t),\qquad x_+(t),
\]
with \(x_\pm(t)\to x_\pm\).  Let
\[
C_\pm(t)=\pi^{-1}(x_\pm(t)).
\]
The normal reduction for the Hopf-basic part is therefore performed along the
slightly displaced fibers \(C_\pm(t)\), not necessarily along the limiting
fibers \(C_\pm\).  The degree \(p\) vertical coefficient is a smooth function of
the fiber.  Since, by hypothesis, this coefficient is nonzero on each limiting
fiber \(C_\pm\), it remains nonzero on \(C_\pm(t)\) for all sufficiently large
\(t\).

The first possible nonconstant vertical contribution therefore occurs in degree
\(p\), with nonzero coefficient on each displaced fiber.  Moreover, any reduced
vertical frequency different from \(0\) and \(\pm p\) has strictly larger heat
weight than the degree \(p\) contribution. Indeed, a linear spectral
contribution with vertical frequency \(\pm mp\), \(m\geq2\), must come from
degree at least \(mp\). Nonlinear contributions involving the degree \(p\)
vertical mode have total heat weight at least twice
\[
\lambda_p-\lambda_2,
\]
and hence are also of strictly higher heat weight. By the spectral tail estimate
and the weighted Morse--Bott expansion lemma, all these contributions are
\(o(e^{-(\lambda_p-\lambda_2)t})\) in \(C^2\).

Hence the reduced function on each Hopf circle has the form
\[
c_t+2\operatorname{Re}(A_t e^{ip\theta})+S_t(\theta),
\]
where
\[
A_t\sim A e^{-(\lambda_p-\lambda_2)t},
\qquad
A\neq0,
\]
and
\[
\frac{\|S_t\|_{C^2}}{|A_t|}\to0.
\]
The Fourier counting lemma then implies that each reduced function has exactly
\[
2p
\]
nondegenerate critical points for all sufficiently large \(t\).

Thus the lifted function \(U(t)\) has
\[
2p+2p=4p
\]
nondegenerate critical points on \(S^3\). Since the action of \(\Gamma\) is free,
these critical points descend in orbits of cardinality \(p\). Therefore the
descended function on
\[
L(p,1)
\]
has exactly
\[
4
\]
critical points.

The indices are obtained by adding the normal Morse--Bott index to the index of
the reduced one-dimensional critical point. The fiber over the base minimum
contributes indices
\[
0,\quad1,
\]
and the fiber over the base maximum contributes indices
\[
2,\quad3.
\]
Thus the four indices are
\[
0,\quad1,\quad2,\quad3.
\]
This proves the proposition.
\end{proof}

\subsection[The case L(2,1) = RP3]{The case \(L(2,1)=\mathbb{RP}^3\)}

We now treat the exceptional case
\[
L(2,1)=\mathbb{RP}^3.
\]
Here the invariant eigenfunctions are precisely the even spherical harmonics on
\(S^3\). The first nonconstant eigenspace is \(E_2\), which identifies with the
trace-free quadratic forms on \(\mathbb R^4\). For a generic quadratic form in
this space, the leading heat term is already Morse on \(\mathbb{RP}^3\). Thus,
in this case, the argument reduces to elementary linear algebra and stability of
Morse functions.

\begin{proposition}[Heat flow for \(\mathbb{RP}^3\)]
Let
\[
L(2,1)=\mathbb{RP}^3.
\]
Let
\[
u(t)=e^{t\Delta}f
\]
be the heat flow, and write
\[
f=f_0+\sum_{k\geq1}f_k
\]
for the spectral decomposition of \(f\). Assume that the \(E_2\)-component of
\(f\) is represented on \(S^3\subset\mathbb R^4\) by a trace-free quadratic form
\[
f_2(x)=x^TAx,
\]
where \(A\) is real symmetric and has four distinct eigenvalues.

Then there exists
\[
T_f>0
\]
such that, for every
\[
t>T_f,
\]
the function
\[
u(t):\mathbb{RP}^3\to\mathbb R
\]
is Morse and has exactly four critical points. Their Morse indices are
\[
0,\quad1,\quad2,\quad3.
\]
\end{proposition}

\begin{proof}
We lift the heat flow to \(S^3\). Constants do not affect critical points, so
we subtract \(f_0\) and normalize by setting
\[
U(t)=e^{\lambda_2t}(u(t)-f_0).
\]
Then
\[
U(t)=f_2+r(t),
\qquad
\|r(t)\|_{C^2}\to0.
\]

The function
\[
f_2(x)=x^TAx
\]
is even, and therefore descends to \(\mathbb{RP}^3\). Since \(A\) is symmetric
with four distinct eigenvalues, 
choose an orthonormal basis
\[
e_1,e_2,e_3,e_4
\]
of eigenvectors, with corresponding eigenvalues
\[
\mu_1<\mu_2<\mu_3<\mu_4.
\]
On \(S^3\), the critical points of \(f_2\) are precisely
\[
\pm e_1,\ \pm e_2,\ \pm e_3,\ \pm e_4.
\]
After passing to the antipodal quotient, these give four critical points on
\[
\mathbb{RP}^3.
\]

The Hessian of \(f_2\) on \(S^3\), restricted to the tangent space at \(e_i\),
has eigenvalues proportional to
\[
\mu_j-\mu_i,
\qquad
j\neq i.
\]
Thus the Morse index of the critical point represented by \(e_i\) is the number
of eigenvalues smaller than \(\mu_i\). Hence the four indices are
\[
0,\quad1,\quad2,\quad3.
\]

In particular, \(f_2\) descends to a Morse function on \(\mathbb{RP}^3\) with
exactly four critical points.

Since
\[
U(t)\to f_2
\]
in \(C^2\), and Morse functions are stable under sufficiently small
\(C^2\)-perturbations, the descended function \(U(t)\) is Morse with the same
number of critical points and the same indices for all sufficiently large
\(t\). Finally,
\[
u(t)-f_0=e^{-\lambda_2t}U(t),
\]
with a positive scalar factor, so \(u(t)\) has the same critical points and
indices as \(U(t)\). This proves the proposition.
\end{proof}

\subsection{Genericity and the main theorem}

We now collect the genericity assumptions used in the three heat-flow
propositions. In each case these assumptions involve only finitely many spectral
components of the initial datum. They are expressed as nonvanishing conditions
for linear or polynomial functions on finite-dimensional eigenspaces, and hence
define open dense subsets.

\begin{lemma}[Generic spectral data]
Let
\[
L(p,q)=S^3/\Gamma,
\qquad
\gamma(\alpha,\beta)=(\zeta\alpha,\zeta^q\beta),
\qquad
\zeta=e^{2\pi i/p},
\]
where
\[
p\geq2,\qquad 1\leq q\leq p/2,\qquad (p,q)=1.
\]
There exists an open dense subset
\[
\mathcal U_{p,q}\subset L^2(L(p,q),\mathbb R)
\]
such that every
\[
f\in\mathcal U_{p,q}
\]
satisfies the genericity hypotheses of the corresponding heat-flow proposition:
the proposition for \(\mathbb{RP}^3\) if \(p=2\), the proposition for
\(L(p,1)\) if \(p\geq3\) and \(q=1\), and the proposition for
\(1<q<p/2\) otherwise.
\end{lemma}

\begin{proof}
We verify openness and density separately in the three cases, following the
hypotheses of the corresponding heat-flow propositions.

First consider
\[
p=2,\qquad q=1.
\]
The spectral projection
\[
L^2(\mathbb{RP}^3)\to E_2
\]
is continuous and surjective. The eigenspace \(E_2\) identifies with the space of
trace-free real symmetric \(4\times4\) matrices. The subset consisting of
matrices with four distinct eigenvalues is open and dense, because the condition
of having a repeated eigenvalue is algebraic of positive codimension. Therefore
its inverse image in \(L^2(\mathbb{RP}^3)\) is open and dense.

Now consider
\[
p\geq3,\qquad q=1.
\]
The condition
\[
f_2\neq0
\]
is open and dense in \(E_2\).  For each nonzero \(f_2\), the corresponding
height function on \(S^2\) has two critical points and hence determines two Hopf
fibers
\[
C_-(f_2),
\qquad
C_+(f_2).
\]
For each sign define the vertical coefficient functional
\[
\Lambda_\pm(f_2):E_p\longrightarrow \mathbb C
\]
by declaring \(\Lambda_\pm(f_2)(h)\) to be the coefficient of \(e^{ip\theta}\)
in the restriction of \(h\) to the Hopf fiber \(C_\pm(f_2)\), with respect to
any positively oriented fiber coordinate.  The nonvanishing of this coefficient
is independent of the chosen origin of \(\theta\).

For every nonzero \(f_2\), the maps \(\Lambda_\pm(f_2)\) are nonzero real-linear
maps.  Indeed, if
\[
C=\{(e^{i\theta}a,e^{i\theta}b):\theta\in\mathbb R/2\pi\mathbb Z\},
\]
then at least one of the restrictions of the invariant degree \(p\) harmonics
\[
\alpha^p,
\qquad
\beta^p
\]
has nonzero \(e^{ip\theta}\)-coefficient on \(C\).  Hence, for fixed nonzero
\(f_2\), the bad set of \(f_p\in E_p\) is
\[
\ker\Lambda_-(f_2)\cup\ker\Lambda_+(f_2),
\]
a finite union of proper real-linear subspaces.  Its complement is open and
dense in \(E_p\).

We now pass from fiberwise density to density in \(E_2\oplus E_p\).  The maps
\[
f_2\longmapsto C_\pm(f_2)
\]
are smooth on \(E_2\setminus\{0\}\), and therefore
\[
(f_2,f_p)\longmapsto \Lambda_\pm(f_2)(f_p)
\]
is continuous.  Thus the good set
\[
\mathcal G_{p,1}
=
\{(f_2,f_p): f_2\ne0,
\ \Lambda_-(f_2)(f_p)\ne0,
\ \Lambda_+(f_2)(f_p)\ne0\}
\]
is open.  It is dense because, given any \((f_2,f_p)\), one first perturbs
\(f_2\) arbitrarily little to make it nonzero, and then, for this fixed
\(f_2\), perturbs \(f_p\) arbitrarily little outside the two proper kernels
\(\ker\Lambda_-(f_2)\) and \(\ker\Lambda_+(f_2)\).  Taking the inverse image of
\(\mathcal G_{p,1}\) under the continuous spectral projection
\[
L^2(L(p,1))\to E_2\oplus E_p
\]
gives an open dense subset of \(L^2(L(p,1),\mathbb R)\).

Finally consider
\[
p\geq3,\qquad 1<q<p/2.
\]
In this case
\[
E_2=\mathbb R\rho.
\]
Thus
\[
f_2=a\rho,\qquad a\neq0,
\]
is an open dense condition on the \(E_2\)-component.

The coefficient of the lowest normal-linear branch near \(C_\alpha\) is a
nonzero real-linear functional on \(E_{q+1}\). For example, it is realized by
the \(\Gamma\)-invariant harmonic monomial
\[
\overline\alpha^{\,q}\beta.
\]
Similarly, the coefficient of the lowest normal-linear branch near \(C_\beta\)
is a nonzero real-linear functional on \(E_{|\mu_\beta|+1}\). It is realized by
the harmonic monomial
\[
N_\beta=
\begin{cases}
\beta^{\mu_\beta}\alpha,
&
\mu_\beta>0,
\\[0.4em]
\overline\beta^{-\mu_\beta}\alpha,
&
\mu_\beta<0.
\end{cases}
\]
This monomial is \(\Gamma\)-invariant because
\[
q\mu_\beta+1\equiv0\pmod p.
\]
Therefore the nonvanishing of the two lowest normal-linear coefficients is an
open dense condition on the corresponding finite-dimensional eigenspaces.

It remains to discuss the first resonant coefficients. These coefficients are
defined after normalizing the heat flow by the leading term
\[
f_2=a\rho.
\]
Equivalently, they are polynomial functions of the normalized spectral
coefficients
\[
a^{-1}f_k
\]
belonging to finitely many eigenspaces. Hence, as functions of the original
finite-dimensional spectral data, they are rational functions whose only
possible denominator is a power of \(a\). Since we are already working on the
open dense set
\[
a\neq0,
\]
the nonvanishing of
\[
\mathfrak A_\alpha(f),
\qquad
\mathfrak A_\beta(f)
\]
is equivalent to the nonvanishing of two nonzero real polynomials obtained by
clearing denominators.

By the nontriviality lemma for the first resonant coefficients, these
polynomials are not identically zero. Therefore the conditions
\[
\mathfrak A_\alpha(f)\neq0,
\qquad
\mathfrak A_\beta(f)\neq0
\]
are open and dense in the relevant finite-dimensional space of spectral
coefficients.

Altogether, in the case \(1<q<p/2\), the genericity assumptions are the
complement of a finite union of proper real algebraic subsets in a finite
dimensional direct sum of eigenspaces. Taking the inverse image under the
continuous spectral projection from
\[
L^2(L(p,q),\mathbb R)
\]
to this finite-dimensional direct sum gives an open dense subset of
\(L^2(L(p,q),\mathbb R)\).

This proves the lemma in all cases.
\end{proof}

\begin{proof}[Proof of Theorem~\ref{thm:main}]
Let
\[
\mathcal U_{p,q}\subset L^2(L(p,q),\mathbb R)
\]
be the open dense subset given by the generic spectral data lemma.

There are three cases.

If \(p=2\) and \(q=1\), then \(L(2,1)=\mathbb{RP}^3\). For every
\(f\in\mathcal U_{2,1}\), the hypotheses of the heat-flow proposition for
\(\mathbb{RP}^3\) are satisfied.
Hence the conclusion follows from that proposition.

If \(p\geq3\) and \(q=1\), then, for every \(f\in\mathcal U_{p,1}\),
the hypotheses of the heat-flow proposition for \(L(p,1)\) are satisfied.
Hence the conclusion follows from that proposition.

Finally, suppose \(p\geq3\) and \(1<q<p/2\). For every
\(f\in\mathcal U_{p,q}\), the hypotheses of the heat-flow proposition for
\(1<q<p/2\) are satisfied.
Hence the conclusion follows from that proposition.

These three cases exhaust all lens spaces in the stated normalization
\(p\geq2\), \(1\leq q\leq p/2\), and \((p,q)=1\). This proves the
theorem.
\end{proof}

\appendix

\section{The arithmetic dominance estimate}\label{app:arithmetic-dominance}

For completeness, we give the proof of the arithmetic lemma used in the
resonant weight computation.

\begin{lemma}[Dominance of the fundamental reduced frequency]
Let
\[
p\geq 3,\qquad 1<a<p/2,\qquad (a,p)=1.
\]
For \(k\geq1\), write
\[
\delta_k=k(k+2)-8.
\]
For each integer \(m\geq1\), define
\[
\Omega_a(m)
=
\min_{d\in\mathbb Z}
\left[
\delta_{|mp-ad|+|d|}
+
|d|\delta_{a+1}
\right].
\]
Then
\[
\Omega_a(1)<\Omega_a(m)
\qquad
\text{for every }m\geq2.
\]
\end{lemma}

\begin{proof}
First, negative values of \(d\) never minimize. Indeed, if \(d<0\), then
\[
|mp-ad|+|d|
=
mp+(a+1)|d|,
\]
and therefore
\[
\delta_{|mp-ad|+|d|}+|d|\delta_{a+1}
>
\delta_{mp},
\]
which is the value obtained at \(d=0\). Hence
\[
\Omega_a(m)
=
\min_{d\geq0}
\left[
\delta_{|mp-ad|+d}
+
d\delta_{a+1}
\right].
\]

For \(x\geq0\), define the real relaxation
\[
F_m(x)
=
\delta_{|mp-ax|+x}
+
x\delta_{a+1}.
\]
Let
\[
H_a=\inf_{x\geq0}F_1(x).
\]
We first prove that
\[
\Omega_a(m)>mH_a
\qquad
\text{for every }m\geq2.
\]

Let \(d\geq0\) be an integer and set
\[
x=\frac{d}{m}.
\]
Then
\[
|mp-ad|+d
=
m\bigl(|p-ax|+x\bigr).
\]
Put
\[
D=|p-ax|+x.
\]
Thus
\[
|mp-ad|+d=mD.
\]
Since
\[
\delta_y=y^2+2y-8,
\]
we have
\[
\delta_{mD}-m\delta_D
=
(m-1)(mD^2+8)>0.
\]
Therefore
\[
\delta_{|mp-ad|+d}+d\delta_{a+1}
=
\delta_{mD}+mx\delta_{a+1}
>
m\delta_D+mx\delta_{a+1}
=
mF_1(x)
\geq
mH_a.
\]
Taking the minimum over all \(d\geq0\), we obtain
\[
\Omega_a(m)>mH_a.
\]

It remains to prove
\[
\Omega_a(1)<2H_a.
\]
Write
\[
p=as+r,
\qquad
s=\left\lfloor\frac pa\right\rfloor,
\qquad
1\leq r\leq a-1.
\]
The inequality \(1\le r\le a-1\) follows from \((a,p)=1\). Since \(p>2a\), we
also have
\[
s\geq2.
\]

Evaluating the discrete expression at \(d=s\), we get
\[
|p-as|+s=r+s.
\]
Hence
\[
\Omega_a(1)
\leq
U_a,
\]
where
\[
U_a=\delta_{s+r}+s\delta_{a+1}.
\]
It is enough to prove
\[
U_a<2H_a.
\]

We now compute \(H_a\). For
\[
0\leq x\leq \frac pa,
\]
set
\[
D=p-(a-1)x.
\]
Then
\[
x=\frac{p-D}{a-1},
\]
and
\[
F_1(x)
=
\delta_D+\frac{p-D}{a-1}\delta_{a+1}.
\]
Since
\[
\delta_{a+1}=a^2+4a-5=(a-1)(a+5),
\]
we obtain
\[
F_1(x)
=
D^2-(a+3)D+(a+5)p-8.
\]
The unconstrained minimum of this quadratic occurs at
\[
D=\frac{a+3}{2}.
\]

For
\[
x\geq \frac pa,
\]
the function \(F_1(x)\) is increasing away from the boundary
\(x=p/a\).

Indeed, for \(x\geq p/a\) we have
\[
|p-ax|+x=(a+1)x-p.
\]
Thus
\[
F_1(x)
=
\delta_{(a+1)x-p}
+
x\delta_{a+1}.
\]
Since
\[
\frac{d}{dx}F_1(x)
=
(a+1)\bigl(2((a+1)x-p)+2\bigr)
+
\delta_{a+1},
\]
and since \(x\geq p/a\) implies \((a+1)x-p\geq p/a>0\), this derivative is
strictly positive. Hence \(F_1\) is increasing on \([p/a,\infty)\).

Therefore
\[
H_a
=
\begin{cases}
(a+5)p-8-\dfrac{(a+3)^2}{4},
&
\dfrac pa\leq \dfrac{a+3}{2},
\\[1.2em]
\left(\dfrac pa\right)^2
-(a+3)\dfrac pa
+(a+5)p-8,
&
\dfrac pa> \dfrac{a+3}{2}.
\end{cases}
\]

We prove
\[
U_a<2H_a
\]
in the two cases.

First suppose
\[
\frac pa\leq \frac{a+3}{2}.
\]
Equivalently,
\[
2as+2r\leq a(a+3).
\]
A direct substitution gives
\[
2H_a-U_a
=
\frac12
\Bigl(
2a^2s-a^2+4ar+12as-6a
-2r^2-4rs+16r-2s^2+6s-25
\Bigr).
\]
Using
\[
1\leq r\leq a-1,
\]
we estimate
\[
-2r^2\geq -2(a-1)^2,
\qquad
-4rs\geq -4(a-1)s,
\]
and we discard the positive terms \(4ar\) and \(16r\). Thus
\[
2(2H_a-U_a)
\geq
2a^2s+8as+10s-2s^2-3a^2-2a-27.
\]
In the present case,
\[
s\leq \frac{a+3}{2}.
\]
The function
\[
G(s)=2a^2s+8as+10s-2s^2
\]
is increasing on the interval
\[
2\leq s\leq \frac{a+3}{2},
\]
because
\[
G'(s)=2a^2+8a+10-4s
\geq
2a^2+6a+4>0.
\]
Therefore
\[
G(s)\geq G(2)=4a^2+16a+12.
\]
Hence
\[
2(2H_a-U_a)
\geq
a^2+14a-15>0
\]
for all \(a\geq2\). Thus
\[
U_a<2H_a
\]
in the first case.

Now suppose
\[
\frac pa> \frac{a+3}{2}.
\]
A direct substitution gives
\[
2H_a-U_a=\frac{N}{a^2},
\]
where
\[
\begin{aligned}
N={}&
a^4s+2a^3r+4a^3s
-a^2r^2-2a^2rs+6a^2r
+a^2s^2-3a^2s-8a^2  \\
&\qquad
+4ars-6ar+2r^2.
\end{aligned}
\]
Again using
\[
1\leq r\leq a-1,
\]
we estimate
\[
-a^2r^2\geq -a^2(a-1)^2,
\]
\[
-2a^2rs\geq -2a^2(a-1)s,
\]
and
\[
-6ar\geq -6a(a-1).
\]
Discarding the remaining positive terms, we get
\[
N
\geq
a^2s(a^2+2a-1)
+
a^2s^2
-a^4+2a^3-15a^2+6a.
\]
Since
\[
s\geq2,
\]
this gives
\[
N
\geq
2a^2(a^2+2a-1)
+
4a^2
-a^4+2a^3-15a^2+6a.
\]
Thus
\[
N
\geq
a^4+6a^3-13a^2+6a
=
a(a-1)(a^2+7a-6).
\]
This is strictly positive for
\[
a\geq2.
\]
Therefore
\[
2H_a-U_a>0
\]
also in the second case.

We have proved
\[
\Omega_a(1)\leq U_a<2H_a.
\]
Moreover,
\[
U_a=\delta_{s+r}+s\delta_{a+1}>0,
\]
because
\[
s\geq2,\qquad r\geq1,
\]
so that
\[
s+r\geq3,
\]
and
\[
\delta_{s+r}>0,\qquad \delta_{a+1}>0.
\]
Hence
\[
H_a>0.
\]
Combining this with
\[
\Omega_a(m)>mH_a
\]
for \(m\geq2\), we obtain
\[
\Omega_a(m)>mH_a\geq2H_a>\Omega_a(1).
\]
Hence
\[
\Omega_a(1)<\Omega_a(m)
\qquad
\text{for every }m\geq2.
\]
\end{proof}

\section*{Disclosure statement}

No potential conflict of interest was reported by the authors.

\section*{Funding}

This research was funded by Universidad EAFIT (Colombia), project ``Study and
Applications of Diffusion Processes of Importance in Health and Computation''
(project code 11740052022).

\section*{Data availability statement}

No data sets are required to verify the results in this article. The numerical
experiments mentioned in the introduction served as a heuristic discovery
mechanism; the proof is entirely analytic.

\section*{AI usage disclosure}

The authors used ChatGPT (OpenAI; GPT-5.5 Thinking) as an AI-assisted tool for
idea exploration, mathematical discussion, LaTeX drafting, editing, and
manuscript preparation. The research program leading to this article predates
the use of AI tools by more than a decade. In the specific case of lens spaces,
the authors had previously observed that, although the leading heat mode need
not be Morse, adding small multiples of suitable eigenfunctions from higher
eigenspaces could produce Morse-minimal functions in the examples under study.
AI-assisted discussions were then used to explore this observation further and
helped formulate the Morse--Bott reduction used in the paper; in particular,
they helped clarify the perturbative role of resonant higher modes such as
\(\operatorname{Re}(\alpha^p+\beta^p)\) in breaking the two critical circles in
the case \(q=1\). Subsequent AI-assisted discussions also helped organize the
hierarchical conical-neighborhood framework suggested by the heat-flow
asymptotics and assisted with exposition, LaTeX editing, and preparation of the
final manuscript. No AI tool was used to generate experimental data, figures, or
synthetic research results. The authors checked all mathematical arguments,
computations, citations, and references, and take full responsibility for the
originality, accuracy, and integrity of the entire article.

\end{document}